\documentclass[11pt]{article}
\usepackage{amsmath,color,amsfonts,amsthm}
\usepackage{amssymb,enumerate,verbatim}
\usepackage{hyperref}

\usepackage{geometry}
\usepackage{graphicx}
\newtheorem{lemma}{Lemma}

\newtheorem{claim}[lemma]{Claim}
\newtheorem{theorem}[lemma]{Theorem}

\newtheorem{definition}[lemma]{Definition}
\newtheorem{conjecture}[lemma]{Conjecture}

\theoremstyle{definition}

\newcommand{\Remark}[1]{}

\title{The rough structure of biclaw-free bipartite graphs}

\date{}

\begin{document}

\title{Hamiltonicity and structure of connected biclaw-free graphs}
\author{
	Alexey Pokrovskiy \thanks{
	Department of Mathematics, 
	University College London.
	Email: \texttt{a.pokrovskiy}@\texttt{ucl.ac.uk}.
	}
	\and
	Xiaoan Yang \thanks{
	School of Mathematical Sciences, Queen Mary University of London.
	Email: \texttt{xiaoan.yang}@\texttt{qmul.ac.uk}.
	The authors thank the UCL Mathematics Department for financially supporting the second author's undergraduate summer project.}
}

\date{}

\maketitle

\begin{abstract}
We show that for sufficiently large $d$, every balanced bipartite, connected biclaw-free graph with minimum degree $\geq d$ is Hamiltonian. This confirms a conjecture of Flandrin, Fouquet, and Li.
\end{abstract}

\section{Introduction}
A Hamiltonian cycle in a graph is a cycle which passes through every vertex. They were introduced by the great Irish mathematician William Rowan Hamilton, and have been extensively studied since then. Generally we are interested in proving that all graphs with some properties contain a Hamiltonian cycle. The most well known theorem like this is Dirac's theorem which says that every $G$ with $\delta(G)\geq |V(G)|/2$ has a Hamiltonian cycle. Here $\delta(G)$ denotes the minimum degree of $G$ i.e. the minimum number of edges that pass through a vertex in $G$.

Dirac's Theorem has a huge number variations and generalizations (see e.g. the survey(s)~\cite{faudree1997claw}). However many of these generalizations have a limitation in that they are about \emph{dense} graphs i.e. about graphs with lots of edges. For sparse graphs, we still have a relatively incomplete understanding of Hamiltonicity. There are various open problems about this. One of the most famous is the following conjecture.
\begin{conjecture}[Matthews, Sumner, \cite{matthews1985longest}]\label{Conjecture_Matthews_sumner}
Every $4$-connected, claw-free graph has a Hamiltonian cycle.
\end{conjecture}
Here ``claw-free'' means that there is no induced claw (a graph with vertices $v,x,y,z$ and edges $vx,vy,vz$). There are a variety of results known about this conjecture --- see the survey~\cite{faudree1997claw} or the book~\cite{li2008hamiltonian}. 

This paper will be about a bipartite variant of the Matthews-Sumner Conjecture. Every bipartite graph with $\Delta(G)\geq 3$ contains an induced claw, so clearly ``claw-free'' needs to be changed to something else in order to make a meaningful variant. Flandrin, Fouquet, Li defined a biclaw to be the graph in Figure~\ref{Figure_biclaw}. A graph $G$ is  ``induced $H$-free'' if it contains no induced subgraphs isomorphic to $H$.
\begin{figure} [h]\label{Figure_biclaw}
  \centering
     \includegraphics[width=0.3\textwidth]{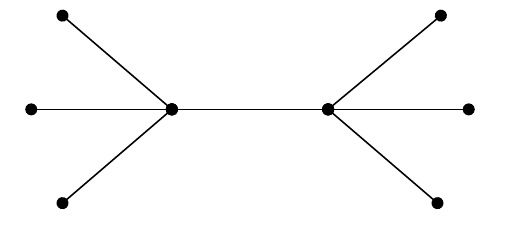}
     \caption{A biclaw $S_{3,3}$.}
\end{figure} 

They posed the following variant of Conjecture~\ref{Conjecture_Matthews_sumner}. 
\begin{conjecture}[Flandrin, Fouquet, Li, \cite{flandrin1994hamiltonicity}]\label{conj_main}
There is a constant $d\in \mathbb{N}$, such that every connected, induced biclaw-free, balanced bipartite graph with $\delta(G)\geq d$ has a Hamiltonian cycle.
\end{conjecture}
Here most of the assumptions (other than biclaw-free) are easily seen to be necessary for Hamiltonicity in one way or another: 
Both ``connected'' and ``balanced bipartite'' (that the parts of the bipartition of $G$ have the same size)  are necessary for Hamiltonicity since bipartite cycles are balanced, connected graphs. The conjecture wouldn't be true without the $\delta(G)\geq \delta$ assumption (e.g. because it would be satisfied by a single path which is not Hamiltonian) --- however Flandrin, Fouquet, Li remark that probably a small value of $\delta$ (like $\delta=6$) should suffice.
Define a $(a,b)$-biclaw, denoted $S_{a,b}$, to be the graph with $V(S_{a,b})=\{x, x_1, \dots, x_a, y, y_1, \dots, y_b\}$ and $E(S_{a,b})=\{xy\}\cup \{xy_1, \dots, x y_b, y x_1, \dots, y x_a\}$. Thus a biclaw in the sense of Flandrin-Fouquet-Li
is  $S_{3,3}$, in our notation. It is very natural to make a strenghtening of Conjecture~\ref{conj_main} where we ask for $G$ to be induced $S_{t,t}$-free rather than  biclaw-free (and letting $d$ depend on $t$).

There has been some progress on the conjecture  Flandrin, Fouquet, Li proved their conjecture assuming $\delta(G) > n/6 + O(1)$. Barraez,  Flandrin,   Li,  and Ordaz strengthened this to $\delta(G)> n/8+O(1)$ and also proved a variant involving dominating cycles instead of Hamiltonian cycles~\cite{barraez1995dominating}. 
Bai and Li~\cite{bai2023connected} proved that every induced $S_{1,3}$-free graph balanced connected bipartite graph  with $\delta(G)\geq 4$ is Hamiltonian.  Lai and Yao proved that biclaw-free graphs with $\delta(G)\geq 5$ have certain properties (called ``collapsible'' and ``supereulerian'')  related to, but weaker than, Hamiltonicity~\cite{lai2006collapsible}. 

In this paper, we prove Conjecture~\cite{flandrin1994hamiltonicity}, as well as well as the natural generalization of it to $(t,t)$-biclaws. We prove 
\begin{theorem}\label{Theorem_main}
For all $t\in \mathbb{N}$, there is a $d\in \mathbb{N}$, such that every connected, induced $S_{t,t}$-free, balanced bipartite graph with $\delta(G)\geq d$ has a Hamiltonian cycle.
\end{theorem}
Somewhat surprisingly, despite proving the conjecture, our proof approach shows that much of the intuition and motivation for the Flandrin-Fouquet-Li Conjecture is incorrect. The primary motivation for the conjecture is that it appears related to the Matthews-Sumner Conjecture, and specifically the area of Hamiltonicity of sparse graphs. However, the main step of our proof consists of proving that connected, induced $S_{t,t}$-free, bipartite graphs with $\delta(G)\geq d(t)$ must actually be \emph{extremely} dense. 

\begin{theorem}\label{Theorem_dense}
For all $t\in \mathbb{N}, \epsilon>0$, there is a $d\in \mathbb{N}$, such that every connected, induced $S_{t,t}$-free, bipartite graph with $\delta(G)\geq d$ and parts of order $n$ has $e(G)\geq (1-\epsilon)n^2$.
\end{theorem}
Getting Hamiltonicity after establishing $e(G)\geq (1-\epsilon)n^2$ is routine, so most of the work in this paper is about proving Theorem~\ref{Theorem_dense}.  In this way, the paper is more about structural graph theory than Hamiltonicity --- the closest result to ours is probably by Scott, Seymour, and Spirkl~\cite{scott2023pure} who proved that for any forest $F$, bipartite induced $F$-free graphs either have $e(G)\geq (1-\epsilon)n^2$ or have  sets of order $\epsilon n$ in each part with no edges between them. 

Proving Theorem~\ref{Theorem_dense} comes down using  Kovari-S\'os-Tur\'an-style arguments to gradually establish more and more structure in $S_{t,t}$-free graphs. 
The proof goes in the following steps: In Section~\ref{Section_diameter} we show that such graphs have small diameter. Then in Section~\ref{Section_maxdeg}, we show that almost all vertices in such graphs have degree $\geq (1-\epsilon)n$. In Section~\ref{Section_Hamiltonian} we prove Hamiltonicity.

\section*{Establishing small diameter}\label{Section_diameter}
We'll use the following version of the Kovari-S\'os-Tur\'an Theorem.
\begin{theorem} [Kovari-S\'os-Tur\'an]\label{thm_KST}
    Let $G$ be a  bipartite graph with parts $A, B$ having no $K_{s,t}$. Then $e(G) < (s-1)^{1/t}(|B|-t+1)|A|^{1-1/t}+(t-1)|A|$.
\end{theorem}
We say that a bipartite graph is $\overline{K_{t,t}}$-free if any two sets of size $t$ on  opposite sides of the bipartition have an edge between them. This is equivalent to the bipartite complement of the graph being  $K_{t,t}$-free. The relevance for us is that if a graph is $S_{t,t}$-free then for any edge $uv$, the bipartite graph between $N(u)$ and $N(v)$ is $\overline{K_{t,t}}$-free. 
\begin{lemma}\label{lem_bihole_free}
There is a function $C_1(t, \epsilon)$ such that the following holds for $t\in \mathbb{N}, \epsilon\in (0,1)$:  
Let $G$ be bipartite with parts $A,B$ which is $\overline{K_{t,t}}$-free having $|B|\geq \epsilon^{-1}t$. Then $A$ contains $\le C_1(t, \epsilon)$ vertices of degree $\le (1-\epsilon)|B|$. 
\end{lemma}
\begin{proof}
Let $C_1(t, \epsilon)=(\epsilon^{-1}(t-1)^{1/t})^{t}$. 
Let $A'$ be the set of vertices in $A$  of  degree $\le (1-\epsilon)|B|$,  and let $H$ be the bipartite complement of $G$ between $A'$ and $B$. We have that $\delta_H(A')\geq \epsilon|B|$ giving $e(G)\geq \epsilon|A'||B|$. Also  $H$ is $K_{t,t}$-free, so Theorem~\ref{thm_KST} gives $e(H)\le (t-1)^{1/t}(|B|-t+1)|A'|^{1-1/t}+(t-1)|A'|\le (t-1)^{1/t}|B||A'|^{1-1/t}+\epsilon|B||A'|/2$. Combining gives $\epsilon|B||A'|/2\le (t-1)^{1/t}|B||A'|^{1-1/t}$ and hence $|A'|\le (2\epsilon^{-1}(t-1)^{1/t})^{t}$.
\end{proof}

Many of our proofs come down to studying the following sets.
\begin{definition}
Let $X$ and $Y$ be two sets of vertices and $\varepsilon \in (0,1)$, define \[S_X^Y (\varepsilon) = \{x\in X \mid d_Y(x)\leq (1-\varepsilon)|Y|\}\]
\end{definition}

Negating this definition gives the following property which we'll use repeatedly.  
$$\text{for } x\in X\setminus S_X^Y (\varepsilon) \text{ we have } N(x)> (1-\epsilon)|Y|$$
We often write ``using the definition of $S_{\ast}^{\ast}$'' to mean that we are using the above property.

The goal of this section is to prove that for any $a,b$ in opposite parts of $G$, the set $S_{N(a)}^{N(b)}(\epsilon)$ is small (see Lemma~\ref{lem_S_bound_general}). This is done in several steps, depending on the how far $a$ and $b$ are from each other. First $d(u,v)=1$.
\begin{lemma}\label{lem_S_AB}
There is a function $C_2(t, \epsilon)$ such that the following holds for $t\in \mathbb{N}, \epsilon\in (0,1)$:  Let $G$ be a bipartite graph which is induced $S_{t,t}$-free. For every edge $ab$, and $A\subseteq N(a), B\subseteq N(b)$ with $|B|\geq C_2(t, \epsilon)$, we have $S_{B}^{A}(\epsilon)\le C_2(t, \epsilon)$.
\end{lemma}
\begin{proof}
Let $C_2(t, \epsilon)=\max(\epsilon^{-1}(t+1), C_1(t+1, \epsilon))$. 
Let $H$ be the subgraph of $G$ between $A$ and $B$. Since $G$ is $S_{t,t}$-free, we know that $H$ is $\overline{K_{t+1,t+1}}$-free (a $\overline{K_{t+1,t+1}}$ in $H$ gives a $\overline{K_{t,t}}$ avoiding the vertices $a,b$. This $\overline{K_{t,t}}$ together with $a,b$ gives an induced $S_{t,t}$). Note $|B|\geq C_2(t, \epsilon)\geq \epsilon^{-1}(t+1)$. Therefore Lemma~\ref{lem_bihole_free} applies in order to show that $A$  has $\le C_1(t+1, \epsilon)\le C_2(t, \epsilon)$ vertices with  $\le (1-\epsilon)|B|$ neighbours in $B$, which is equivalent to $S_{A}^{B}(\epsilon)\le C_2(t, \epsilon)$.
\end{proof}

We isolate the following special case.
\begin{lemma}
There is a function $C_2(t, \epsilon)$ such that the following holds: Let $G$ be a bipartite graph with parts $A,B$, with $\delta(G)\geq C_2(t, \epsilon)$, and which is induced $S_{t,t}$-free. For every edge $ab$ we have $S_{N(a)}^{N (b)}(\epsilon)\le C_2(t, \epsilon)$.
\end{lemma}
\begin{proof}
Let $A=N(a), B=N(b)$ and apply Lemma~\ref{lem_S_AB}
\end{proof}
Next we prove a version of the above when $a$ and $b$ are distance $3$ from each other.
\begin{lemma}\label{lem_S_dist_3}
Let $0<\epsilon<1$. 
There is a function $C_3(t, \epsilon)$ such that the following holds: Let $G$ be a bipartite graph with parts $A,B$, with $\delta(G)\geq C_3(t, \epsilon)$, and which is induced $S_{t,t}$-free. For every length $3$ path $x_1x_2x_3x_4$ we have $S_{N(x_1)}^{N (x_4)}(\epsilon)\le C_3(t, \epsilon)$.
\end{lemma}
\begin{proof}
Let $C_3(t, \epsilon)=4C_2(t, \epsilon/4)$.
Let $S= S_{N(x_1)}^{N (x_4)}(\epsilon)\subseteq N(x_1)$, supposing for contradiction that $|S|\geq C_3(t, \epsilon)$.

 Note $|S|, |N(x_3)|, |N(x_4)|\geq C_2(t, \epsilon/4)$. 
 By Lemma~\ref{lem_S_AB} we have $S_{N(x_2)}^{S}(\epsilon/4), S_{N(x_2)}^{N(x_3)}(\epsilon/4), S_{N(x_3)}^{N(x_4)}(\epsilon/4)\le C_2(t, \epsilon/4)<C_3(t, \epsilon)/2\le  |N(x_2)|/2$. Since $|N(x_2)\setminus (S_{N(x_2)}^{S}(\epsilon/4)\cup S_{N(x_2)}^{N(x_3)}(\epsilon/4))|\geq |N(x_2)|-|S_{N(x_2)}^{S}(\epsilon/4)|-|S_{N(x_3)}^{N(x_4)}(\epsilon/4)|>0$, this gives that there's a vertex $v\in N(x_2)\setminus (S_{N(x_2)}^{S}(\epsilon/4)\cup S_{N(x_2)}^{N(x_3)}(\epsilon/4))$. This gives that $|N(v)\cap S|\geq (1-\epsilon/4)|S|$ and  $|N(v)\cap N(x_3)|\geq (1-\epsilon/4)|N(x_3)|\geq |N(x_3)|/2$.  Recalling $S_{N(x_3)}^{N(x_4)}(\epsilon/4)<N(x_3)/2$, this gives that there is a vertex $w\in N(v)\cap N(x_3)\setminus S_{N(x_3)}^{N(x_4)}(\epsilon/4)$.
 
Now let $A=N(v)\cap S$, $B=N(w)\cap N(x_4)$. Note $|A|\ge (1-\epsilon/4)|S|\geq |S|/2\geq  C_3(t, \epsilon)/2>C_2(t, \epsilon/4)$. Since $w\not\in  S_{N(x_3)}^{N(x_4)}(\epsilon/4)$, the definition of $S_{\ast}^{\ast}$ gives $|B|\geq (1-\epsilon/4)|N(x_4)|\ge |N(x_4)|/2\ge  C_3(t, \epsilon)/2 \geq C_2(t, \epsilon/4)$. Thus Lemma~\ref{lem_S_AB} applies to the edge $vw$ to give $|S_{A}^B(\epsilon/4)|\le C_2(t, \epsilon/4)<|A|$. This gives a vertex $a\in A\setminus  S_{A}^B(\epsilon/4)$ which, using the definition of $S_{\ast}^{\ast}$, must have $|N(a)\cap B|\geq (1-\epsilon/4)|B|\geq (1-\epsilon/4)^2|N(x_4)|> (1-\epsilon)|N(x_4)|$. But also $a\in A\subseteq  S=S_{N(x_1)}^{N (x_4)}(\epsilon)$ which gives  $|N(a)\cap B|\le  (1-\epsilon)|N(x_4)|$, which is a contradiction.
\end{proof}

Almost the same proof works for paths of length $5$.
\begin{lemma}\label{lem_S_dist_5}
Let $0<\epsilon<1$. 
There is a function $C_4(t, \epsilon)$ such that the following holds: Let $G$ be a bipartite graph with parts $A,B$, with $\delta(G)\geq C_4(t, \epsilon)$, and which is induced $S_{t,t}$-free. For every length $5$ path $x_1x_2x_3x_4x_5x_6$ we have $S_{N(x_1)}^{N (x_6)}(\epsilon)\le C_4(t, \epsilon)$.
\end{lemma}
\begin{proof}
Let $C_4(t, \epsilon)=\max(4C_3(t, \epsilon/4), 4C_2(t, \epsilon/4))$.
Let $S= S_{N(x_1)}^{N (x_4)}(\epsilon)\subseteq N(x_1)$, supposing for contradiction that $|S|\geq C_4(t, \epsilon)$.

 Note $|S|, |N(x_3)|, |N(x_4)|\geq C_2(t, \epsilon/4), C_3(t, \epsilon/4)$. 
 By Lemma~\ref{lem_S_AB} we have $S_{N(x_2)}^{S}(\epsilon/4),  S_{N(x_5)}^{N(x_6)}(\epsilon/4)\le C_2(t, \epsilon/4)<C_4(t, \epsilon)/2\le  |N(x_2)|/2$, and by Lemma~\ref{lem_S_dist_3}, we have $S_{N(x_2)}^{N(x_5)}(\epsilon/4)\le C_3(t, \epsilon/4)<C_4(t, \epsilon)/2\le  |N(x_2)|/2$. Since $|N(x_2)\setminus (S_{N(x_2)}^{S}(\epsilon/4)\cup S_{N(x_2)}^{N(x_3)}(\epsilon/4))|\geq |N(x_2)|-|S_{N(x_2)}^{S}(\epsilon/4)|-|S_{N(x_3)}^{N(x_4)}(\epsilon/4)|>0$, this gives that there's a vertex $v\in N(x_2)\setminus (S_{N(x_2)}^{S}(\epsilon/4)\cup S_{N(x_2)}^{N(x_5)}(\epsilon/4))$. This gives that $|N(v)\cap S|\geq (1-\epsilon/4)|S|$ and  $|N(v)\cap N(x_5)|\geq (1-\epsilon/4)|N(x_5)|\geq |N(x_5)|/2$.  Recalling $S_{N(x_5)}^{N(x_6)}(\epsilon/4)<N(x_5)/2$, this gives that there is a vertex $w\in N(v)\cap N(x_5)\setminus S_{N(x_5)}^{N(x_6)}(\epsilon/4)$.
 
Now let $A=N(v)\cap S$, $B=N(w)\cap N(x_6)$. Note $|A|\ge (1-\epsilon/4)|S|\geq |S|/2\geq  C_4(t, \epsilon)/2>C_2(t, \epsilon/4)$. Since $w\not\in  S_{N(x_5)}^{N(x_6)}(\epsilon/4)$, the definition of $S_{\ast}^{\ast}$ gives $|B|\geq (1-\epsilon/4)|N(x_6)|\ge |N(x_6)|/2\ge  C_4(t, \epsilon)/2 \geq C_2(t, \epsilon/4)$. Thus Lemma~\ref{lem_S_AB} applies to the edge $vw$ to give $|S_{A}^B(\epsilon/4)|\le C_2(t, \epsilon/4)<|A|$. This gives a vertex $a\in A\setminus  S_{A}^B(\epsilon/4)$ which, using the definition of $S_{\ast}^{\ast}$, must have $|N(a)\cap B|\geq (1-\epsilon/4)|B|\geq (1-\epsilon/4)^2|N(x_4)|> (1-\epsilon)|N(x_4)|$. But also $a\in A\subseteq  S=S_{N(x_1)}^{N (x_6)}(\epsilon)$ which gives  $|N(a)\cap B|\le  (1-\epsilon)|N(x_6)|$, which is a contradiction.
\end{proof}

Finally, we are ready to bound the diameter of biclaw-free graphs.
\begin{lemma}\label{lem_bound_diameter}
There is a function $C_5(t)$ such that the following holds: Let $G$ be a bipartite connected graph  with $\delta(G)\geq C_5(t)$, and which is induced $S_{t,t}$-free. Then $G$ has diameter $\le 5$.
\end{lemma}
\begin{proof}
Let $C_5(t)=\max(4C_3(t, 1/4), 10)$. 
Suppose for contradiction that there are vertices with $d(x_0, x_6)=6$, and let $x_0x_1x_2x_3x_4x_5x_6$ be a shortest path between them. By Lemma~\ref{lem_S_dist_3}, we have $S_{N(x_3)}^{N(x_0)}(1/4)$, $S_{N(x_3)}^{N(x_6)}(1/4)\le  C_3(t, 1/4)< N(x_3)/2$, giving a vertex $v\in N(x_3)\setminus (S_{N(x_3)}^{N(x_0)}(1/4)\cup S_{N(x_3)}^{N(x_6)}(1/4))$. By definition of $S_{\ast}^{\ast}$, we have $|N(v)\cap N(x_0)|\geq  (1-1/4)|N(x_0)|>0$ and $|N(v)\cap N(x_6)|\geq  (1-1/4)|N(x_6)|>0$ giving edges $vu$ and $vw$ with $u\in N(x_0), w\in N(x_6)$. Now $x_0 u v w x_6$ is a length $4$ path from $x_0$ to $x_6$ contradicting $d(x_0, x_6)=6$.
\end{proof}

The following lemma summarizes everything in this section.
\begin{lemma}\label{lem_S_bound_general}
There is a function $C_6(t, \epsilon)$ so that the following is true for $t\in \mathbb{N}, \epsilon\in (0,1)$.  Let $G$ be a bipartite connected graph with parts $A,B$, with $\delta(G)\geq C_6(t, \epsilon)$, and which is induced $S_{t,t}$-free. For every $a\in A, b\in B$, we have $S_{N(a)}^{N (b)}(\epsilon)\le C_6(t, \epsilon)$.
\end{lemma}
\begin{proof}
Let $C_6(t,\epsilon)=\max(C_2(t,\epsilon), C_3(t,\epsilon), C_4(t,\epsilon), C_5(t))$. 
Since $a,b$ are in different parts of $G$, $d(a,b)$ is odd. By 
By Lemma~\ref{lem_bound_diameter} $d(a,b)\le 5$, giving that $d(a,b)=1,3,$ or $5$. Depending on which of these happens, the lemma comes from Lemma~\ref{lem_S_AB}, \ref{lem_S_dist_3} or \ref{lem_S_dist_5}.
\end{proof}

\section{Establishing density}\label{Section_maxdeg}

For a bipartite graph $G$ with parts $X,Y$, we use $\Delta_X, \Delta_Y$ to denote the maximum degrees of vertices in $X$ and $Y$ respectively.  
\begin{definition}\label{def_U}
Let $X$ be a set of vertices of a graph $G$ and $0 < \varepsilon < 1$. Define
\[U_X(\varepsilon) = \{x\in X \mid d(x)\leq (1-\varepsilon)\Delta_X\}\]
where $\Delta_X = max\{d(x), x\in X\}$.
\end{definition}
We'll ultimately prove that $\Delta_X, \Delta_Y$ are close to $n$ (Lemma~\ref{thm_max_degree}) and that $U_X(\varepsilon), U_Y(\varepsilon)$ are almost equal to $X,Y$ (Lemma~\ref{lem_U}). First we bound degrees into these sets. In this, and other lemmas in this section $C_6(t, \epsilon)$ will always refer to the function that comes from Lemma~\ref{lem_S_bound_general}.
\begin{lemma} \label{lem_U}
There is a function $C_6(t, \epsilon)$ so that the following is true for $t\in \mathbb{N}, \epsilon\in (0,1)$. Let $G$ be a connected balanced bipartite graph with parts $X,Y$, $\delta(G)\geq C_6(t, \epsilon)$ and with no induced $S_{t,t}$. For all $x \in X$, $y \in Y$ we have $d_{U_Y(\varepsilon)}(x) \leq C_6(t, \epsilon)$ and $d_{U_X(\varepsilon)}(y) \leq C_6(t, \epsilon)$.
\end{lemma}
\begin{proof}
We'll just prove $d_{U_Y(\varepsilon)}(x) \leq C_6(t, \epsilon)$ --- and$d_{U_X(\varepsilon)}(y) \leq C_6(t, \epsilon)$ can be proved by exchanging all instances of $X$ and $Y$. 
Let $y' \in Y$ such that $d(y') = \Delta_Y$. For all $x \in X$, we have
\[N(x) \cap U_Y(\varepsilon) \subseteq S_{N(x)}^{N(y')}(\varepsilon).\]
This is because if $v \in N(x) \cap U_Y(\varepsilon)$, then $v \in N(x)$ and $d_{N(y')}(v) \leq d(v) \leq (1-\varepsilon)\Delta_Y =  (1-\varepsilon)d(y')$ i.e. $d_{N(y')}(v) \leq (1-\varepsilon)|N(y')|$, so $v \in S_{N(x)}^{N(y')}(\varepsilon)$. Thus, by Lemma \ref{lem_S_bound_general},
\[d_{U_Y(\varepsilon)}(x) = |N(x) \cap U_Y(\varepsilon)| \leq |S_{N(x)}^{N(y')}(\varepsilon)| \leq C_6(t, \epsilon).\]
\end{proof}
We next show that these sets $U_{X}(\varepsilon), U_{Y}(\varepsilon)$ are small.
\begin{lemma}\label{size_UX}
There is a function $C_6(t, \epsilon)$ so that the following is true for $t\in \mathbb{N}, \epsilon\in (0,1)$. 
Let $G$ be a connected balanced bipartite graph with parts $X,Y$ or order $n$, $\delta(G)\geq C_6(t, \epsilon)$ and with no induced $S_{t,t}$. 
Then 
    \[|U_{X}(\varepsilon)|, |U_{Y}(\varepsilon)| \leq \frac{C_6(t, \epsilon)}{\delta(G)}n.\] 
\end{lemma}
\begin{proof}
    By Lemma \ref{lem_U}, we have  \[e(U_X(\varepsilon), Y) = \sum_{y \in Y} d_{U_X(\varepsilon)}(y) \leq C_6(t, \epsilon)|Y| = C_6(t, \epsilon)n.\]
    On the other hand,
    \[e(U_X(\varepsilon), Y) = \sum_{u \in U_X(\varepsilon)} d(u) \geq \delta(G)|U_X(\varepsilon)|.\]
Combine both inequalities, we obtain the desired result. The proof for $|U_{Y}(\varepsilon)|$ is the same by exchanging the roles of $X$ and $Y$.
\end{proof}

Next we study common neighbourhoods of vertices.
\begin{lemma}\label{lem_common_neighbour_vertex}
There is a function $C_6(t, \epsilon)$ so that the following is true for $t\in \mathbb{N}, \epsilon\in (0,1)$. 
Let $G$ be a connected balanced bipartite graph with parts $X,Y$ or order $n$, $\delta(G)\geq 3C_6(t, \epsilon)$ and with no induced $S_{t,t}$. For any vertices $a, b$ in the same part of $G$, there is a vertex $c$ in the same part such $|N(a)\setminus N (c)|\leq \epsilon|N(a)|$, $|N(b)\setminus N (c)|\le \epsilon|N(b)|$.
\end{lemma}
\begin{proof}
Let $v$ be an arbitrary vertex in the opposite part to $a,b$. By Lemma~\ref{lem_S_bound_general}, we have $|S_{N(v)}^{N(a)}|, |S_{N(v)}^{N(b)}|\le C_6(t, \epsilon)$ giving at least $|N(v)|-|S_{N(v)}^{N(a)}|-|S_{N(v)}^{N(b)}|\geq C_6(t, \epsilon)$ choices of a vertex $c$ outside both of these sets. By definition of $S_{\ast}^{\ast}$, we have $|N(a)\cap  N (c)|\geq (1-\epsilon)|N(a)|$, $|N(b)\cap  N (c)|\ge (1-\epsilon)|N(b)|$ which implies the lemma.
\end{proof}

Applying the above to two vertices outside $U_X(\epsilon)$ gives the following.
\begin{lemma}\label{lemma_common_nhood}
There is a function $C_6(t, \epsilon)$ so that the following is true for $t\in \mathbb{N}, \epsilon\in (0,1)$. 
Let $G$ be a connected balanced bipartite graph with parts $X,Y$ or order $n$, $\delta(G)\geq 3C_6(t, \epsilon)$ and with no induced $S_{t,t}$.
For any vertices  $x,x'\not\in U_X(\epsilon)$, we have $|N(x)\cap N(x')|\geq (1-10\epsilon) \Delta_X$. 
\end{lemma}
\begin{proof}
Suppose for contradiction that $|N(x)\cap N(x')|<(1-10\epsilon) \Delta_X$. 
Since  $x,x'\not\in U_X(\epsilon)$, we have $|N(x)|, |N(x')|\geq (1-\epsilon)\Delta_X$. By the inclusion-exclusion principle, we have 
$$|N(x)\cup N(x')|\geq |N(x)|+|N(x')|-|N(x)\cap  N(x')|\geq (1-\epsilon)\Delta_X+(1-\epsilon)\Delta_X- (1-10\epsilon) \Delta_X=(1+8\epsilon)\Delta_X.$$  By Lemma~\ref{lem_common_neighbour_vertex}, there is a vertex $c$ with $|N(x)\setminus N (c)|\leq \epsilon|N(x)|\le \epsilon\Delta_X$, $|N(x')\setminus N (c)|\le \epsilon|N(x')|\le \epsilon\Delta_X$. This gives $|N (c)|\geq |N(x)\cup N(x')|-|N(x)\setminus N (c)|-|N(x')\setminus N (c)|\geq (1+8\epsilon)\Delta_X-\epsilon\Delta_X-\epsilon\Delta_X>\Delta_X$, which is a contradiction. 
\end{proof}

We now prove that the maximum degree in the graph is large.
\begin{lemma}[Maximum degree lemma]\label{thm_max_degree}
There is a function $C_6(t, \epsilon)$ so that the following is true for $t\in \mathbb{N}, \epsilon\in (0,1)$. 
Let $G$ be a connected balanced bipartite graph with parts $X,Y$ or order $n$, $\delta(G)\geq 3C_6(t, \epsilon)$ and with no induced $S_{t,t}$.
 We have \[\Delta_X, \Delta_Y \geq \left(1-\frac{C_6(t, \epsilon)}{\delta(G)}\right)(1-10\varepsilon)n\].
\end{lemma}

\begin{proof}
We'll just prove the inequality for $\Delta_Y$, and the one for $\Delta_X$ can be done by exchanging the roles of $X$ and $Y$. 
Let $v \in U_X^c(\varepsilon)$ (for example take $v$ to be a vertex of degree $\Delta_X$). We have
\begin{align*}
    e(U_X^c(\varepsilon), N(v)) & = \sum_{x \in U_X^c(\varepsilon)} d_{N(v)}(x)\ge |U_X^c(\varepsilon)|(1-10\varepsilon)\Delta_X
    \ge\left(n-\frac{C_6(t, \epsilon)}{\delta(G)}n\right)(1-10\varepsilon)\Delta_X
\end{align*}
Here the first inequality if from Lemma~\ref{lemma_common_nhood}, while the second one is by Lemma~\ref{size_UX}. 
 Now for an upper bound,
\[e(U_X^c(\varepsilon), N(v)) = \sum_{y \in N(v)} d_{U_X^c(\varepsilon))}(y) \leq |N(v)|\Delta_Y \leq \Delta_X\Delta_Y \]
Combining both inequalities we obtain the desired result.
\end{proof}

We summarize everything in this section by the following which is slightly stronger than Theorem~\ref{Theorem_dense}

\begin{lemma} \label{lem_density}
Let $\alpha, \epsilon>0, t\in \mathbb{N}$, and let $d$ be sufficiently large. Let $G$ be a connected balanced bipartite graph with parts $X,Y$ or order $n$, $\delta(G)\geq d$ and with no induced $S_{t,t}$. Then $|U_X(\alpha)|, |U_Y(\alpha)|\leq \epsilon n$ and vertices outside these sets have degree $\ge (1-20\epsilon)n$.
\end{lemma}
\begin{proof}
Let $d=\epsilon^{-1}4C_6(t, \epsilon)$. By Lemma~\ref{thm_max_degree} we have $\Delta_X, \Delta_Y\geq (1-\frac{C_6(t, \epsilon)}{d})(1-10\epsilon)n\geq (1-\frac{C_6(t, \epsilon/40)}{\epsilon^{-1}4C_6(t, \epsilon)})(1-10\epsilon)n= (1-\epsilon/4)(1-10\epsilon)n\geq (1-11\epsilon)n$.
By Lemma~\ref{size_UX}, we have $|U_X(\epsilon)|, |U_Y(\epsilon)|\le \frac{C_6(t, \epsilon)}{\delta(G)}n\le \frac{C_6(t, \epsilon)}{\epsilon^{-1}4C_6(t, \epsilon)}n=\epsilon n/4$. All vertices outside these sets have degrees $\geq (1-\epsilon)\Delta_X\geq (1-\epsilon)(1-11\epsilon)n\geq (1-20\epsilon)n$ or $\geq (1-\epsilon/2)\Delta_Y\geq (1-\epsilon)(1-11\epsilon)n\geq (1-20\epsilon)n$ as required.
\end{proof}

\section{Hamiltonicity}\label{Section_Hamiltonian}
Here we prove Hamiltonicity of biclaw-free graphs. It is easy to find a very long cycle using Lemma~\ref{lem_density}. However to get a Hamilton cycle we also need to cover vertices outside $U_X(\epsilon), U_Y(\epsilon)$. For this we use the following lemma which guarantees edges leaving $U_X(\epsilon), U_Y(\epsilon)$.
\begin{lemma} \label{matching_lemma}
There is a function $C_6(t, \epsilon)$ so that the following is true for $t\in \mathbb{N}, \epsilon\in (0,1)$. 
Let $G$ be a connected balanced bipartite graph with parts $X,Y$ or order $n$, $\delta(G)\geq 3C_6(t, \epsilon)$ and with no induced $S_{t,t}$.
 \begin{itemize}
 \item Let $U_{X}(\varepsilon) = \{u_1,\dots,u_m\}$. There exists a set of distinct vertices $\{v_i^-, v_i^+: i=1, \dots, m\}\subseteq U_{Y}^c(\varepsilon)$ such that for all $i$ we have $u_iv_i^-, u_iv_i^+$ edges.
  \item Let $U_{Y}(\varepsilon) = \{v_1,\dots,v_m\}$. There exists a set of distinct vertices $\{u_i^-, u_i^+: i=1, \dots, m\}\subseteq U_{X}^c(\varepsilon)$ such that for all $i$ we have $v_iu_i^-, v_iu_i^+$ edges.
 \end{itemize}
\end{lemma}
\begin{proof}
We'll just prove the first bullet point --- the second one follows from the same proof, exchanging the roles of $X$ and $Y$. 
Construct $V'$ with two copies of each vertex in $U_{X}(\varepsilon)$ i.e. $V' = \{u_i^+, u_i^- \mid N(u_i) = N(u_i^+) = N(u_i^-) \text{ and } i=1, \dots ,m.\}$. If we can show there is a matching $M = \{u_i^+v_i^+, u_i^-v_i^- \mid v_i^{\pm} \in U_{Y}^c(\varepsilon)\}$ i.e.  a matching from $V'$ to $U_{Y}^c(\varepsilon)$ covering $V'$, then we are done. By Hall's theorem, it suffices to show for any $S \subseteq V'$, we have $|N(S) \cap U_{Y}^c(\varepsilon)| \geq |S|$. We have
\begin{align*}
    e(S, U_{Y}^c(\varepsilon) \cap N(S)) & = \sum_{y \in N(S) \cap U_{Y}^c(\varepsilon)} d_{S}(y) 
     \leq \sum_{y \in N(S) \cap U_{Y}^c(\varepsilon)} d_{V'}(y) 
    = \sum_{y \in N(S) \cap U_{Y}^c(\varepsilon)} 2d_{U_{X}(\varepsilon)}(y)  \\
   &\leq 2C_6(t, \epsilon)|N(S) \cap U_{Y}^c(\varepsilon)|.
\end{align*}
Here the first inequality uses $S\subseteq V'$, and the second one comes from Lemma \ref{lem_U}. 
We also have,
\[e(S, U_{Y}^c(\varepsilon) \cap N(S)) = e(S, U_{Y}^c(\varepsilon))= \sum_{x \in S} d_{U_{Y}^c(\varepsilon)}(x) = \sum_{x \in S} (d(x) - d_{U_{Y}(\varepsilon)}(x))\]
Again by Lemma \ref{lem_U},
\[e(S, U_{Y}^c(\varepsilon) \cap N(S)) \geq (\delta(G) - C_6(t, \epsilon))|S|.\]
Hence, combining both inequalities we get,
   \[|U_{Y}^c(\varepsilon) \cap N(S)| \geq \frac{(\delta(G) - C_6(t, \epsilon))}{2C_6(t, \epsilon)}|S| \geq |S|\]
\end{proof}
Exchanging the roles of $X$ and $Y$ we also have:

We'll use the following standard bipartite analogue of Dirac's Theorem.
\begin{theorem}[Moon and Moser, \cite{moon1963hamiltonian}]\label{MoonMoser}
Let $G$ be a bipartite graph with parts of order $n$ and $\delta(G)> n/2$. Then $G$ is Hamiltonian.
\end{theorem}

We now prove our main result. 
\begin{proof}[Proof of Theorem \ref{Theorem_main}]
Let $G$ be bipartite into $X$ and $Y$ with $|X|=|Y|=n$. Pick $\delta(G)$ large enough so that $|U_X(0.01), U_Y(0.01)|\leq 0.0001 n$, and so that vertices outside these sets have degree $\ge 0.9n$ (which is possible by  Lemma~\ref{lem_density}), and also so that Lemmas~\ref{matching_lemma} applies.
\begin{claim}\label{Lemma_one_short_path}
Let $u,v\not\in U_X(0.01)\cup  U_Y(0.01)$, and $S\subseteq V(G)\setminus \{u,v\}$, with $|S|\leq 0.01n$. Then there is a path from $u$ to $v$ of length $\leq 3$ which doesn't use any vertices of $S$.
\end{claim}
\begin{proof}
Then there are three cases depending on where $u$ and $v$ are:

Suppose that $u,v\in X$.  We have $|N(u)\setminus S|\geq |N(u)|-|S|\geq 0.9n-0.01n=0.89n$. Similarly $|N(v)\setminus S|\geq 0.89n$, which implies that $(N(u)\setminus S)\cap(N(u)\setminus S)\neq \emptyset$ (since both of these sets are contained in $Y$ which has size $n$). Pick any $z\in (N(u)\setminus S)\cap(N(u)\setminus S)$, and let $P=uzv$ to get a path from $u$ to $v$ of length $2$ avoiding $S$.

Suppose that $u,v\in Y$. This is the same as the previous paragraph, exchanging the roles of $X$ and $Y$.

Suppose that $u\in X$ and $v\in Y$. As before, we have that $|N(u)\setminus S|\geq 0.89n$. Let $u'\in N(u)\setminus S$, and set $S'=S\cup \{u\}$. As above, we get $|N(u')\setminus S'|\ge |N(u')|- |S'| \geq 0.88n$, and similarly $|N(v)\setminus S'|\geq 0.88n$.  This again gives us a vertex $z\in (N(u)\setminus S)\cap(N(u)\setminus S)$. Now the path $P=uu'zv$ is from $u$ to $v$, has length $3$ and avoids $S$.
\end{proof}
Using this we can cover $U_{X}(0.01), U_{Y}(0.01)$.
\begin{claim} \label{lem_path}
There is a path $P$ covering $U_{X}(0.01) \cup U_{Y}(0.01)$    starting in $U_{Y}^c(0.01)$ and ending in $U_{X}^c(0.01)$, such that $|P| \leq 0.01n$.
\end{claim}
\begin{proof}
Let  $U_{X}(0.01)\cup U_Y=\{u_1, \dots, u_m\}$, noting $m\le 0.0002$. 
From Lemma~\ref{matching_lemma} we have a set of distinct vertices $\{u_i^-, u_i^+: i=1, \dots, m\}\subseteq U_{X}^c(0.01)\cup U_{Y}^c(0.01)$. such that for all $i$ we have $u_iu_i^-, u_iu_i^+$ edges. Noting $|\{u_i^-, u_i^+: i=1, \dots, m\}|< U_{X}^c(0.01), U_{Y}^c(0.01)$, pick some $u_0^+\in U_{Y}^c(0.01)$, $u_{m+1}^-\in U_{X}^c(0.01)$, outside $\{u_i^-, u_i^+: i=1, \dots, m\}$. 
We will construct disjoint paths $P_1, \dots, P_{m-1}$ of length $\leq 3$ with $P_i$ going from $u_i^+$ to $u_{i+1}^-$. We do this one by one for $i=0, \dots, m$. Having already constructed $P_0, \dots, P_{i-1}$, we build $P_i$ by applying Lemma~\ref{Lemma_one_short_path} with $u=u_{i}^+, v=u_{i+1}^-$, $S=\bigcup_{j=1}^{i-1}V(P_i)\cup \bigcup_{j={i+1}}^{m}\{u_j^-, u_j^+\}$ (noting that this has $|S|\leq 5m\leq 0.01n$). Now the following is a path satisfying the lemma 
$$P:=u_0^+P_0u_1^-u_1u_1^+P_1u_2^-u_2u_2^+P_2u_3^-u_3u_3^+P_3\dots u_{m-1}^-u_{m-1}u_{m-1}^+P_{m-1}u_m^-u_mu_m^+P_m u_{m+1}^-.$$
Finally, we have $|P| \leq 5m \leq 0.01n$
\end{proof}

Let $H = G \setminus P$.
Consider some $v\in V(H)$. Since $v\not\in P$ and $U_X, U_Y\subseteq V(P)$, we have $d_H(v)\geq d_G(v)-|V(P)|\geq 0.9n-0.01n\geq 0.6n$ i.e. we have $\delta(H)\geq 0.6n$. Also note that since $P$ is a path in a bipartite graph from $X$ to $Y$, so $|P\cap X| = |P\cap Y|$. Thus, $H$ is a balanced bipartite graph.     Now by Theorem~\ref{MoonMoser}, $H$ has a Hamiltonian cycle $C$. Note $|C|= 2n-|P|\geq 1.99n$ and so $|C|/2-0.1n\geq 3|C|/8$.

Let the starting vertex of $P$ be $y' \in U_{Y}^c(0.01)$ and the ending vertex of $P$ be $x' \in U_{X}^c(0.01)$. 
\begin{claim} \label{prop_la}
    There exists $u \in N(x')$, $v \in N(y')$ such that $uv \in E(C)$.
\end{claim}
\begin{proof}[Proof of Proposition \ref{prop_la}]
        Since $C$ is a Hamiltonian cycle of $H$, its edges admit a perfect matching, say $M = \{x_1y_1,\dots ,x_{|C|/2}y_{|C|/2} \mid x_iy_i \in E(C)\}$.  Since $y'\not\in U_{Y}(0.01)$, we know that $d(y')\geq 0.9n$ and so  $y'$ is connected to $\geq |C|/2-0.1n\geq 3|C|/8$ of the vertices   $x_1,\dots ,x_{|C|/2}$. Similarly $x'$ is connected to $\geq |C|/2-0.1n\geq 3|C|/8$ of the vertices   $y_1,\dots ,y_{|C|/2}$. Thus there is some index $i$ with $x'y_i$ and $y'x_i$ both edges (and so $x_iy_i$ satisfies the lemma).     
    \end{proof}

Now writing $C$ as $C = c_1, \dots, uv, \dots, c_1$, the sequence 
$c_1, \dots, u,P,v, \dots, c_1$ gives a Hamiltonian cycle in $G$.
\end{proof}

	\bibliography{biclaw}

\begin{thebibliography}{1}

\bibitem{bai2023connected}
Y.~Bai and B.~Li.
\newblock Connected k-factors in bipartite graphs.
\newblock {\em Discrete Mathematics}, 346(1):113174, 2023.

\bibitem{barraez1995dominating}
D.~Barraez, E.~Flandrin, H.~Li, and O.~Ordaz.
\newblock Dominating cycles in bipartite biclaw-free graphs.
\newblock {\em Discrete Mathematics}, 146(1-3):11--18, 1995.

\bibitem{faudree1997claw}
R.~Faudree, E.~Flandrin, and Z.~Ryj{\'a}{\v{c}}ek.
\newblock Claw-free graphs—a survey.
\newblock {\em Discrete Mathematics}, 164(1-3):87--147, 1997.

\bibitem{flandrin1994hamiltonicity}
E.~Flandrin, J.-L. Fouquet, and H.~Li.
\newblock Hamiltonicity of bipartite biclaw-free graphs.
\newblock {\em Discrete Applied Mathematics}, 51(1-2):95--102, 1994.

\bibitem{lai2006collapsible}
H.-J. Lai and X.~Yao.
\newblock Collapsible biclaw-free graphs.
\newblock {\em Discrete mathematics}, 306(17):2115--2117, 2006.

\bibitem{li2008hamiltonian}
M.~Li.
\newblock Hamiltonian theory in claw-free graphs.
\newblock {\em Discrete Mathematics Research Progress, editor: Kenneth B.
  Moore}, pages 139--211, 2008.

\bibitem{matthews1985longest}
M.~M. Matthews and D.~P. Sumner.
\newblock Longest paths and cycles in k1, 3-free graphs.
\newblock {\em Journal of graph theory}, 9(2):269--277, 1985.

\bibitem{moon1963hamiltonian}
J.~Moon and L.~Moser.
\newblock On hamiltonian bipartite graphs.
\newblock {\em Israel Journal of Mathematics}, 1:163--165, 1963.

\bibitem{scott2023pure}
A.~Scott, P.~Seymour, and S.~Spirkl.
\newblock Pure pairs. iv. trees in bipartite graphs.
\newblock {\em Journal of Combinatorial Theory, Series B}, 161:120--146, 2023.

\end{thebibliography}
	\bibliographystyle{abbrv}
\end{document}